\documentclass[10pt,reqno]{amsart}
\usepackage{amsfonts}
\usepackage[a4paper,inner=1in,outer=1in,vmargin=1in,marginparwidth=1in]{geometry}
\usepackage{amsmath,amsthm,amsfonts,amssymb}
\usepackage{bbm}
\usepackage{cite}
\usepackage{color}
\usepackage{enumerate}
\usepackage{tikz-cd}
\usepackage{mathtools}
\usepackage{microtype}  % Subtle adjustments to prevent overflows
\allowdisplaybreaks[4]  % More aggressive breaks in aligns/multlines

\usepackage{url}  % For breaking URLs if any appear
\DeclareUrlCommand{\citeurl}{\urlstyle{same}}  % Style citations like URLs if needed

\makeatletter
\@namedef{subjclassname@2020}{\textup{2020} Mathematics Subject Classification}
\makeatother
\def\Q{\mathbb{Q}}
\def\C{\mathbb{C}}
\def\N{\mathbb{N}}
\def\Z{\mathbb{Z}}
\def\R{\mathbb{R}}

\newtheorem{theo}{Theorem}[section]
\newtheorem{defi}[theo]{Definition}
\newtheorem{lemm}[theo]{Lemma}
\newtheorem{remark}[theo]{Remark}
\newtheorem{coro}[theo]{Corollary}
\newtheorem{prop}[theo]{Proposition}

\numberwithin{equation}{section}
\allowdisplaybreaks
\begin{document}
	\title{Whittaker Modules for W type Cartan Lie superalgebras}
	
	 \author{Vyacheslav Futorny, Santanu Tantubay}
\address{Shenzhen International Center for Mathematics, Southern University of Science and Technology, Shenzhen, China}
\email{vfutorny@gmail.com}
\email{1mathsantanu@gmail.com}
	 
	 \begin{abstract}
 		We consider the category of Whittaker modules for the Lie superalgebra $W_{m,n}$ of vector fields on $\C^{(m|n)}$. For any $\mathbf{a}\in \C^m$ we show the equivalence between the blocks 
        $\Omega_{\mathbf a}^{\widetilde{W}_{m,n}}$ 
        of the category of $(AW)_{m,n}$-Whittaker modules 
        %of type $\phi_\mathbf{a}$
        with finite-dimensional Whittaker subspaces and the category of finite-dimensional modules over certain Lie subsuperalgebra  $T_{m,n}$ of $(AW)_{m,n}$ (and also of $\mathfrak{gl}{(m,n)})$. Then we apply the covering technique to study Whittaker $W_{m,n}$-modules  and describe simple modules in the category $\Omega_{\mathbf a}^{{W}_{m,n}}$ of such modules 
        %of type $\phi_\mathbf{a}$ 
        with finite-dimensional Whittaker subspaces and with non-singular ${\mathbf a}$. 
	\end{abstract}

	\maketitle
	
	\section{Introduction}
Whittaker modules play an important role in representation theory since their introduction for complex semisimple Lie algebras. In \cite{AP}, Arnal and Pinczon defined Whittaker modules for $\mathfrak{sl}_2$.  In \cite{K}, Kostant defined Whittaker modules for all complex semisimple Lie algebras, and showed that for a non-singular Whittaker map, simple Whittaker modules are in bijective correspondence with maximal ideals of the center of the universal enveloping algebra of a finite-dimensional simple Lie algebra. 
Whittaker modules were studied for various algebras, e.g. 
for quantum groups in [\cite{O}, \cite{S}, \cite{XGZ}, \cite{XZ}]; for Virasoro algebras in [\cite{OW}, \cite{OW1}, \cite{GLZ}, \cite{LPX}, \cite{LWZ}]; for Heisenberg algebra in \cite{C}; for Witt type Lie algebras in \cite{ZL}; for Weyl algebras and generalized Weyl algebras in [\cite{BO}, \cite{CW}]; for affine Lie algebras in [\cite{ALZ}, \cite{CGLW}, \cite{GL}, \cite{CJ}]; for classical Lie superalgebras in \cite{Ch}. In \cite{BM}, the authors studied Whittaker modules in a more general setting for algebras with  a triangular decomposition. 
\medskip

In \cite{CSS} the authors introdued BGG category, classified simple objects and studied homological properties there. In this paper we consider Whittaker modules for Lie superalgebras $W_{m,n}$. Previously, such  modules were classified for $W_{m,0}$ in \cite{ZL}. 

\medskip
 
 Let $A_{m,n}$ be the tensor product of polynomial algebra in $m$-variables $\C[t_1,\dots, t_m]$ and supercommutative exterior algebra $\Lambda(n)$ in $n$ odd variable $\xi_1\dots, \xi_n$ over the complex numbers $\C$. Let $W_{m,n}$ be the space of all superderivations of $A_{m,n}$, which is a Lie superalgebra. The subspaces $\mathfrak{h}_{m,n}:=\text{Span}\{t_i\frac{\partial}{\partial t_i},\; \xi_j\frac{\partial}{\partial \xi_j}: 1\leq i\leq m,\; 1\leq j\leq n\}$ and $\triangle_{m,n}=\text{Span}\{\frac{\partial}{\partial t_i}, \frac{\partial}{\partial \xi_j}: 1\leq i\leq m,\; 1\leq j\leq n\}$ are two commutative subalgebras of $W_{m,n}$. A $W_{m,n}$ module $M$ is called a weight module if $\mathfrak{h}_{m,n}$ is diagonalizable on $M$. Weight modules with finite-dimensional weight spaces are called Harish-Chandra Modules.  Tensor $W_{m,n}$-modules $F(P,M)$ were constructed in \cite{XW} 
 for any simple module $P$ over the Weyl superalgebra and any simple modules $M$ over the general linear Lie superalgebra. 
 In particular, all simple bounded  modules were classified. Harish-Chandra modules for $W_{m,0}$ are classified in \cite{GS},  every such nonzero module is eiher module of tensor fields or a simple subquotient of a tensor module. Harish-Chandra modules Lie superalgebras of $W$ type were classified in \cite{BFIK}, \cite{XL} (see also \cite{BF} for the Lie algebra case). \\

 Let $\widetilde{W}_{m,n}=A_{m,n}\rtimes W_{m,n}$ be the extended Witt superalgebra. We say that a $\widetilde{W}_{m,n}$-module $M$ is a $(AW)_{m,n}$ module if the action of $A_{m,n}$ is associative.
 We see that $(W_{m,n},\triangle_{m,n})$ and $(\widetilde{W}_{m,n},\triangle_{m,n})$ are Whittaker pairs in the sense of \cite{BM} (in the  super case). A $W_{m,n}$ (or $(AW)_{m,n}$)-module $M$ is called Whittaker module if $\triangle_{m,n}$ acts locally finitely on $M$. For a Whittaker $W_{m,n}$-module $M$, a vector $v\in M$ is called a Whittaker vector if $\frac{\partial}{\partial t_i}v=a_iv$ for some ${\mathbf a}=(a_1, \ldots, a_m)\in \C^m$, $1\leq i\leq m$, and $\frac{\partial}{\partial \xi_j}v=0$, $1\leq j\leq n$. Denote by $\Omega_a^{W_{m,n}}$ (resp. $\Omega_{\mathbf a}^{ {AW}_{m,n}}$)  the category of $W_{m,n}$ (resp. $(AW)_{m,n}$ )-modules on which $\frac{\partial}{\partial t_i}-a_i,\; \frac{\partial}{\partial \xi_j}$ acts locally nilpotently for any $i\in \{1,\dots, m\}$ and $j\in \{1,\dots n\}.$ Modules of these categories are Whittaker modules of type ${\mathbf a}$.

 A vector $\mathbf{a}\in \C^m$ will be called  non-singular if $a_i\neq 0$ for all $i\in \{1,2,\dots ,m\}$.
 In this paper we will classify all simple objects in $\Omega_{\mathbf{a}}^{ AW_{m,n}}$ and $\Omega_{\mathbf{a}}^{W_{m,n}}$ for non-singular ${\mathbf a}\in \C^m$. We call the corresponding Whittaker modules non-singular.

 The paper is organized is as follows. In Section \ref{NP} we recall  definitions of  Cartan Lie superalgebras of W type, Weyl superalgebras and also the important concept of the weighting functor which sends any $W_{m,n}$-module (resp. $(AW)_{m,n}$-module)  to a weight  $W_{m,n}$-module (resp. $(AW)_{m,n}$-module). We introduce the notion of Whittaker modules for Witt superalgebras and Weyl superalgebras. In particular, we define the universal Whittaker modules for Witt Lie superalgebras. 
 In Section \ref{KR} we  review some important results from \cite{XL} and \cite{LX}, which play crucial role in our paper. In Section \ref{WM}, we define an $(AW)_{m,n}$-module structure on $T(A^{\mathbf a},V)$  for any $\mathfrak{gl}(m,n)$-module $V$, then we prove that any simple non-singular Whittaker module for $(AW)_{m,n}$  is of the form $T(A^{\mathbf a},V)$ for some finite-dimensional simple $\mathfrak{gl}(m,n)$-module $V$. In Section \ref{W}, given any Whittaker module over a Witt superalgebra of non-singular type, we define a Whittaker module over the extended Witt superalgebra of the same type, which is known as the cover of the original module. Using this cover we prove our main result (Theorem \ref{MT}):

 \begin{theo}
     Any simple non-singular Whittaker module  for the Witt superalgebra $W_{m,n}$  is isomorphic to a simple quotient of the module of  tensor fields $T(A^{\mathbf a}, V)$ for some finite-dimensional simple $\mathfrak{gl}(m,n)$-module $V$ and some $\mathbf{a}=(a_1,a_2,\dots, a_m)\in (\C^\times)^m$.
 \end{theo}

  Together with the results of \cite{XW}, the theorem provides a classification of all simple non-singular Whittaker $W_{m,n}$-modules.

    \section{Notations and Preliminaries}\label{NP}

	We denote $\Z,\; \Z_+,\; \N,\; \Q,\; \R,\; \C$ by the set of integers, positive integers, natural numbers, rational numbers and real numbers respectively. All vector spaces and algebras are taken over $\C$. \\ 
 A vector superspace $V$ is a vector space together with a $\Z_2$-gradation, $V=V_{\bar{0}}\oplus V_{\bar{1}}$. The parity of a homogeneous element $a\in V_i$ is denoted by $|a|=i,\; i\in \Z_2$. An element in $V_{\bar{0}}$ is called even and an element in $V_{\bar{1}}$ is called odd. \\ 
	 A superalgebra $\mathbf{A}$, is a vector superspace $\mathbf{A}=A_{\bar{0}}\oplus A_{\bar{1}}$  equipped with a bilinear multiplication satisfying $A_iA_j\subset A_{i+j}$ for $i,j\in \Z_2$. All subalgebras and ideals are understood in $\Z_2$-graded sense. Also, a module $M$ over a superalgebra $A$ is also always understood in $\Z_2$-graded sense.    
	 Any algebra can be thought of as superalgebra by assuming its odd part as zero.

\medskip

    A Lie superalgebra $\mathfrak{g}$ is a $\mathbb Z_2$-graded space $\mathfrak{g}=\mathfrak{g}_{\bar{0}}\oplus \mathfrak{g}_{\bar{1}}$ with a bilinear multiplication $[-,-]$ satisfying:
		\begin{enumerate}
			\item $[a,b]=-(-1)^{|a||b|}[b,a]$,
			\item $[a,[b,c]]=[[a,b],c]+(-1)^{|a||b|}[b,[a,c]],$
		\end{enumerate}
		for all homogeneous elements $a,b,c\in \mathfrak{g}$.

        \medskip

    A module $M$ over an associative superalgebra or Lie superalgebra $\mathfrak{g}$ is said to be strictly simple if it does not have any $\mathfrak{g}$ invariant subspaces (need not be $\Z_2$-graded) except $0$ and $M$. It is easy to see that any strictly simple module is always simple. 

  \subsection{The general linear Lie superalgebra}
Let $\mathfrak{gl}(m,n)=\mathfrak{gl}(m,n)_{\bar{0}} \oplus \mathfrak{gl}(m,n)_{\bar{1}}$ be the general linear Lie superalgebra consisting of all $(m+n)\times (m+n)$ complex matrices of the block form
\begin{equation}\label{sm}
  \begin{pmatrix}
      a & b\\
      c & d
  \end{pmatrix},  
\end{equation}
where $a,\; b,\; c$ and $d$ are respectively $m\times m,\; m\times n,\; n\times m$ and $n\times n$ matrices. The even subalgebra $\mathfrak{gl}(m,n)_{\bar{0}}$ consists of matrices of the form (\ref{sm}) with $b=c=0$ and the odd subspace $\mathfrak{gl}(m,n)_{\bar{1}}$ consists of matrices of the form \ref{sm} with $a=d=0$. In particular
$$\mathfrak{gl}(m,n)_{\bar{0}}=\text{span}\{E_{i,j}: i,j\in \{1,2,\dots, m\}\; \text{or}\; i,j \in \{m+1,\dots, m+n\}\}$$ and 
 $$\mathfrak{gl}(m,n)_{\bar{1}}=\text{span}\{E_{i,m+j}, E_{m+j,i}: i,\in \{1,2,\dots, m\} \;\text{and}\; j \in \{1,\dots, n\}\}$$

	\subsection{Witt Superalgebra} 
 Let $A_{m,n}$ be the tensor superalgebra of polynomial algebra $\C[t_1,\dots t_n]$ in $m$ even variables $t_1,\dots, t_n$ and exterior algebra  $\Lambda(n)$ in $n$ odd variables $\xi_1,\dots ,\xi_n$.

    \medskip
	
	For $\alpha=(\alpha_1,\dots, \alpha_m)\in \Z_+^m$, we write $t^\alpha:=t_1^{\alpha_1}\dots t_m^{\alpha_m}$ and for $i_1,\dots i_k\in \{\bar{1},\cdots ,\bar{n}\},$ we write $\xi_{i_1,\dots, i_k}=\xi_{i_1}\xi_{i_2}\cdots \xi_{i_k} $. Also for any $I=\{i_1,\dots, i_k\}\subset \bar{n}$, we denote $\underline{I}=(l_1,\dots, l_k)$ if $\{i_1,\dots, i_k\}=\{l_1,\dots, l_k\}$ and $l_1<\dots < l_k$. We denote $\xi_{I}:=\xi_{l_1,\dots, l_k}$ and for $j\in I$, we define $I(j)=(\xi_j$'s position in $\xi_I)-1$. 

    \medskip
    
	The space of all superderivations $W_{m,n}$ on $A_{m,n}$ forms a Lie superalgebra known as the Witt superalgebra. The Witt superalgebra has a standard basis:
	\[\{t^{\alpha}\xi_I\frac{\partial}{\partial t_i},\; t^{\alpha}\xi_I\frac{\partial}{\partial \xi_j}|\alpha\in \Z^m_+, I\subset \{1,2,\dots, n\},\; 1\leq i\leq m,\; 1\leq j\leq n\}\]
	with the brackets given by:
    
	\[[t^{\alpha}\xi_I\frac{\partial}{\partial t_i},t^{\beta}\xi_J\frac{\partial}{\partial t_j}]=\beta_it^{(\alpha+\beta-e_i)}\xi_I\xi_J\frac{\partial}{\partial t_j}-\alpha_jt^{(\alpha+\beta-e_j)}\xi_I\xi_J\frac{\partial}{\partial t_i}\]
    
	\[[t^{\alpha}\xi_I\frac{\partial}{\partial t_i},t^{\beta}\xi_J\frac{\partial}{\partial \xi_j}]=\beta_it^{(\alpha+\beta-e_i)}\xi_I\xi_J\frac{\partial}{\partial \xi_j}-(-1)^{|I|(|J|-1)}(-1)^{I(j)}\delta_{j,I}\xi_{J}\xi_{I\setminus{j}}\frac{\partial}{\partial t_i}\]
    
	\[[t^{\alpha}\xi_I\frac{\partial}{\partial \xi_i},t^{\beta}\xi_J\frac{\partial}{\partial \xi_j}]=(-1)^{J(i)}\delta_{i,J}t^{\alpha+\beta}\xi_I\xi_{J\setminus\{i\}}\frac{\partial}{\partial \xi_j}-(-1)^{I(j)}(-1)^{(|I|-1)(|J|-1)}\delta_{j,I}\xi_J\xi_{I\setminus \{j\}}\frac{\partial}{\partial \xi_i}.\]
    
	We see that $A_{m,n}$ is a $W_{m,n}$-module under the super derivation action of $W_{m,n}$. Denote the Lie superalgebra $W_{m,n} \ltimes A_{m,n}$ by $\widetilde{W_{m,n}}$.
    
\begin{defi}
    A  $\widetilde{W_{m,n}}$ module $M$ is called $(AW)_{m,n}$ module, if $A_{m,n}$ acts associatively on $M$.
 \end{defi}
    \medskip
    
	We  will write $\triangle_{m,0}=\text{Span}\{\frac{\partial}{\partial t_1}, \dots, \frac{\partial}{\partial t_m}\}$, $\triangle_{0,n}=\text{Span}\{\frac{\partial}{\partial \xi_1},\dots ,\frac{\partial}{\partial \xi_n}\}$,  $\mathfrak{h}_{m,0}=\text{Span}\{t_i\frac{\partial}{\partial t_i}: 1\leq i\leq m\}$ and $\mathfrak{h}_{0,n}=\text{Span}\{\xi_j\frac{\partial}{\partial \xi_j}: 1\leq j\leq n\}$ and $\mathfrak{h}_{m,n}=\mathfrak{h}_{m,0}\oplus \mathfrak{h}_{0,n}$.

     We see that $W_{m,n}$ is ad-diagonalizable with respect to $\mathfrak{h}_{m,n}$.

\medskip

	\subsection{Weyl superalgebra}
	The Weyl superalgebra $K_{m,n}$ is the complex simple associative superalgebra generated by
	$t_1,\dots ,t_m,\xi_1,\dots ,\xi_n,\frac{\partial}{\partial t_1},\dots , \frac{\partial}{\partial t_m},\frac{\partial}{\partial \xi_1},\dots \frac{\partial}{\partial \xi_n}$ subject to the relations: 
	$[\frac{\partial}{\partial t_i},\frac{\partial}{\partial t_j}]=[t_i,t_j]=[t_i,\xi_k]=[\frac{\partial}{\partial \xi_k},\frac{\partial}{\partial \xi_l}]=[\frac{\partial}{\partial t_i},\frac{\partial}{\partial \xi_k}]=0$, $\frac{\partial}{\partial t_i}.t_j-t_j.\frac{\partial}{\partial t_i}=\delta_{i,j},$ $\xi_k.\xi_l=-\xi_l\xi_k,\; \frac{\partial}{\partial \xi_k}.\xi_l+\xi_l.\frac{\partial}{\partial \xi_k}=\delta_{k,l}$ for $1\leq i,j\leq m,\; 1\leq k,l\leq n$.
	
	\subsection{Weight modules}
     A $W_{m,n}$-module $M$ 
  is said to be a weight module if $M=\oplus_{\lambda \in \C^m }M_{\lambda}$, where $M_{\lambda}=\{v\in M: t_i\frac{\partial}{\partial t_i}.v=\lambda_i v, : 1\leq i\leq m\}$.
	A $W_{m,n}$-module $M$ is called a strong weight module if $M=\oplus_{\lambda \in \C^m,\,
 \mu \in \C^n}M_{(\lambda,\; \mu)}$, where $M_{(\lambda,\; \mu)}=\{v\in M: t_i\frac{\partial}{\partial t_i}.v=\lambda_i v,\; \xi_j\frac{\partial}{\partial \xi_j}v=\mu_jv: 1\leq i\leq m,\; 1\leq j\leq n\}$. We note that in \cite{LX} the weight modules are equipped with a diagonalizable action of $\mathfrak{h}_{m,n}$, while our weight modules have a diagonalizable action of $\mathfrak{h}_{m,0}$.

	\subsection{Weighting functor}
	In this subsection we recall the notion of the weighting functor introduced in \cite{JN} and  \cite{ZL}. For any $r\in \C^m$, suppose that $I_r$ is the maximal ideal of $\mathcal{U}(\C \;\mathfrak{h}_{m,0})$ generated by 
	$$h_1-r_1,\dots ,h_m-r_m.$$
	For any $W_{m,n}$-module $M$ and $r\in \C^m$, set $M^r:=M/I_rM$ and denote
	$$\mathfrak{W}(M)=\bigoplus_{r\in \Z^m}M^r.$$
	
	\begin{prop}
		The space $\mathfrak{W}(M)$ becomes a weight $W-$module under the following action:
		$$t^r\xi_I\frac{\partial}{\partial t_i}.(v+I_sM):=t^r\xi_I \frac{\partial}{\partial t_i}v+I_{r+s-e_i}M,$$
		$$t^r\xi_I\frac{\partial}{\partial \xi_j}.(v+I_sM):=t^r\xi_I\frac{\partial}{\partial \xi_j}v+I_{r+s}M,$$
		where $v\in M,\; r\in \Z_+^m,\; I\subseteq \{1,2,\dots, n\},\; s\in \Z^m,\; 1\leq i\leq m, 1\leq j\leq n.$
	\end{prop}

 \subsection{Whittaker modules}
 
 For any Lie (super)algebra $\mathfrak{g}$ and a Lie (super)subalgebra $\eta$ of $\mathfrak{g}$, we say that $(\mathfrak{g},\eta)$ is a Whittaker pair \cite{BM} if $\eta$ is a quasi-nilpotent subalgebra of  $\mathfrak{g}$ {(i.e  $\cap_{k=0}^{\infty}  
 \eta_k=0$, where $\eta_0:=\eta$ and $\eta_k:=[\eta_{k-1},\eta]$) and $\eta$ acts (by adjoint action) locally nilpotently on $\mathfrak{g}/\eta.$ 
 We see that $(W_{m,n}, \triangle_{m,n})$ and $(\widetilde{W}_{m,n}, \triangle_{m,n})$ are Whittaker pairs.  
 For any Whittaker pair $(\mathfrak{g},\mathfrak{\eta})$, a $\mathfrak{g}$-module $M$ is called Whittaker module if $\mathfrak{\eta}$ acts locally finitely on M. Denote by $\mathcal{W}_{\mathfrak{g}}^{\mathfrak{\eta}}$  the category of all Whittaker modules for the pair $(\mathfrak{g},\eta)$.

 \medskip

 For a Whittaker pair $(\mathfrak{g}, \mathfrak a)$ and for a Lie superalgebra homomorphism $\phi:\mathfrak a  \rightarrow \C$, we define $M_{\phi}=\mathcal{U}(\mathfrak{g})\otimes_{\mathcal{U}(\mathfrak a)} \C_{\phi}$, where $\C_{\phi}$ is the one dimensional module of $\mathfrak a$ given by $\phi$. We see that $M_{\phi}$ is a Whittaker module. 

 \medskip
 
 A $W_{m,n}$-module $M$ is called a Whittaker module if the action of $\triangle_{m,n}$ on $M$ is locally finite.
	
	For any $\mathbf{a}=(a_1,a_2,\dots, a_m)\in \C^m$, we  define a Lie superalgebra homomorphism $\phi_{\mathbf{a}}:\triangle_{m,n} \rightarrow \C$ such that $\phi_{\mathbf{a}}(\frac{\partial}{\partial t_i})=a_i$ and $\phi_{\mathbf{a}}(\frac{\partial}{\partial \xi_j})=0$  for any $1\leq i\leq m$ and $1\leq j\leq n$. 
	A Whittaker module $M$ is of type $\phi_{\mathbf{a}}$ (or simply of type ${\mathbf a}$) if for any $v\in M$ there exists $k\in \N$ such that $(x-\phi_{\mathbf{a}}(x))^kv=0$ for all $x\in \triangle_{m,n}$. Define the subspace $\text{Wh}_{\mathbf{a}}(M)=\{v\in M| xv=\phi_{\mathbf{a}}(x)v,\; \forall \;x\in \triangle_{m,n}\}$ of $M.$ An element of $\text{Wh}_{\mathbf{a}}(M)$ is called a Whittaker vector.
	Also define $\widetilde{\text{Wh}_{\mathbf{a}}(M)}
	=\{v\in M: \frac{\partial}{\partial t_i} v=a_iv; 1\leq i\leq m\}$, elements of this space are called generalized Whittaker vectors.

    \medskip
    
	Denote by $\Omega_{\mathbf{a}}^{W_{m,n}}$  the subcategory consisting of  Whittaker $W_{m,n}$-modules $M$ of type $\phi_{\mathbf{a}}$ such that $\text{dim} \text{Wh}_{\mathbf{a}}(M) <\infty. $ Similarly, we define Whittaker $(AW)_{m,n}$-modules of type $\phi_{\mathbf{a}}$. Let $\Omega_{\mathbf{a}}^{ AW_{m,n}}$ be the category consisting of Whittaker $(AW)_{m,n}$-modules $M$ of type $\phi_{\mathbf{a}}$ such that $\text{dim} \text{Wh}_{\mathbf{a}}(M)<\infty$.

	\section{Known Results}\label{KR}
 In this section we shall collect some important technical results which will play a crucial role in the description of simple Whittaker modules for the extended Witt algebra $(AW)_{m,n}$.
 
  A module over an associative superalgebra (or Lie superalgebra) is simple if it does not contain any non-zero proper $\Z_2$-graded submodules. A module M over an associative algebra (or Lie superalgebra) $A$  is called strictly simple if it does not have any $A$-invariant subspaces (need not be $\Z_2$-graded) except $0$ and $M$.
  
  We recall the following result  on tensor modules over tensor superalgebras (\cite[Lemma 2.2]{XL}), see also \cite{CFR}.
	\begin{lemm}
		Suppose $B, B^\prime$ be any two unital associative superalgebras with countable basis for $B^\prime, R=B\otimes B^\prime$, $M^\prime$ be a strictly simple $B^\prime$-module and $M$ be a $B$-module Then
		  $M\otimes M^\prime$ is simple $R$ module if and only if $M$ is simple $B$-module.  
		 \end{lemm}
 
	We will use the notions of \cite{LX} here. 
	For convenience denote $A_{m,n}, W_{m,n}, \widetilde{W}_{m,n},$  $ \triangle_{m,n}, K_{m,n}$, $(AW)_{m,n}$ by $A, W, \widetilde{W},\triangle, K$, $AW$ respectively. Recall that a $\widetilde{W}$-module $M$ is $AW$-module if the action of $A$ on $M$ is associative.

    For any Lie (super)algebra $\mathfrak{g}$, let $\mathcal{U}(\mathfrak{g})$ be the universal enveloping algebra of $\mathfrak{g}$. By the Poincare-Birkhoff-Witt theorem, $\mathcal{U}(\widetilde{W})=\mathcal{U}(A)\mathcal{U}(W)$. Let $\mathcal{J}$ be the left ideal of $\mathcal{U}(\widetilde{W})$ generated by 
	\[\{t^0-1,t^\alpha\xi_I\cdot t^\beta\xi_J-t^{\alpha+\beta}\xi_I\xi_J|\alpha,\;\beta\in \Z^m_+, I,J\subset \{1,\dots, n\}\}\]
	One sees that $\mathcal{J}$ is an ideal of $\mathcal{U}(\widetilde{W})$. Let $\bar{U}$ be the quotient algebra $\mathcal{U}(\widetilde{W})/\mathcal{J}$. Identifying $A$ and $W$ with their images in $\bar{U}$, we see that $\bar{U}=A \mathcal{U}(W)$.
    The subspace $AW$ is a Lie supersubalgebra of $\bar{U}$ with the Lie bracket defined as follows:
    \[[a.x,b.y]=ax(b).y-(-1)^{|a.x||b.y|}by(a).x+(-1)^{|x||b|}ab.[x,y],\]
    where $a,b\in A$ and $x,y\in W$.

    \medskip
	For any $\alpha\in \Z^m_+,\; I\subset \{1,2,\dots, n\},\; \partial \in \{\frac{\partial}{\partial t_1}, \dots \frac{\partial}{\partial t_m}, \frac{\partial}{\partial \xi_1}\dots \frac{\partial}{\partial \xi_n} \}$, let 
	\[X_{\alpha,I, \partial}=\sum_{\substack{0\leq \beta\leq \alpha\\ J\subset I}}(-1)^{|\beta|+|J|+\tau(J,I\setminus J)}\binom{\alpha}{\beta}t^\beta\xi_J\cdot t^{\alpha-\beta}\xi_{I\setminus J}\partial.\]
	Denote $T=\text{Span}\{X_{\alpha,I, \partial}|\alpha\in \Z^m_+,\; I\subseteq \{1,2,\dots, n\},\; \partial \in \{\frac{\partial}{\partial t_1}, \dots \frac{\partial}{\partial t_m}, \frac{\partial}{\partial \xi_1}\dots \frac{\partial}{\partial \xi_n},\; |\alpha|+|I|>0 \}\}$.  Then $T$ is a Lie subsuperalgebra of $AW$ by  \cite[Lemma 3.2]{LX}. Let $\mathfrak{m}$ be the maximal ideal of $A$ generated by $t_1,t_2,\dots, t_m, \xi_1, \dots, \xi_n$. Then $\mathfrak{m}\triangle $ is a Lie subsuperalgebra of $W$. We  recall

	\begin{lemm}\label{SI}
    \begin{itemize}
    \item (\cite[Lemma 3.5]{LX})
		There is an associative superalgebra isomomorphism  $$\pi_1:K \otimes \mathcal{U}(T) \rightarrow \bar{U}$$ given by $\pi_1(x\otimes y)=xy$ for all $x\in K, y\in \mathcal{U}(T)$. 
        \item (\cite[Lemma 3.6]{LX})
  The linear map $\pi_2: \mathfrak{m}\triangle \rightarrow T$ defined by $\pi_2(t^\alpha\xi_I\partial)=X_{\alpha, I, \partial}$, for all $t^\alpha\xi_I\in \mathfrak{m}, \partial\in  \{\frac{\partial}{\partial t_1}, \dots \frac{\partial}{\partial t_m}, \frac{\partial}{\partial \xi_1}\dots \frac{\partial}{\partial \xi_n}\}$,  is an isomorphism of Lie superalgebras.
        \end{itemize}
	\end{lemm}

	From the above lemma we see that $\bar{U}\cong K_{m,n}\otimes \mathcal{U}(\mathfrak{m}\triangle)$.\\
 
	\begin{lemm}[\cite{LX}, Lemma 3.7]
		\begin{enumerate}
		   
			\item $\mathfrak{m}\triangle/\mathfrak{m}^2\triangle\cong \mathfrak{gl}(m,n)$.
			\item Suppose that $V$ is a simple weight $\mathfrak{m}\triangle$-module. Then $\mathfrak{m}^2\triangle V=0$, and  hence $V$ is a simple weight $\mathfrak{gl}(m,n)$-module.
             \end{enumerate}
		 
	\end{lemm}
	Note that  the above lemma holds for a finite-dimensional simple $\mathfrak{m}\triangle$-module $V$. Therefore,  any finite-dimensional simple module over $T$ is a $\mathfrak{gl}(m,n)$-module.

\section{Whittaker $AW$-modules}\label{WM}
Our  goal now is to describe all simple Whittaker $AW$-modules using the characterization of finite-dimensional simple $T$-modules.

Let ${\mathbf a}\in \mathbb C^m$. 	A $K$-module $M$ is called a Whittaker module of type $\mathbf{a}$ if for any $v\in M$ there exists $k\in M$ such that  $(x-\phi_{\mathbf{a}}(x))^kv=0$ for all $x\in \triangle$. Let $\mathcal{V}_{\mathbf{a}}^{K}$ be the category consisting of finite length Whittaker $K$-modules of type $\mathbf{a}.$\\
	
	Consider   the superalgebra automorphism $\sigma_{\mathbf{a}}$ of $K$ define as follows:
	\[t_i\mapsto t_i,\; \frac{\partial}{\partial t_i}\mapsto \frac{\partial}{\partial t_i}+a_i,\; \xi_j \mapsto \xi_j,\; \frac{\partial}{\partial \xi_j} \mapsto \frac{\partial}{\partial \xi_j}\]
	Now, twist the $K$-module $A_{m,n}$ by $\sigma_{\mathbf{a}}$ and denote the new $K$-module by $A^{\mathbf{a}}$, where the explicit action of $K$ on $A^{\mathbf{a}}$ is given by
	\[X\cdot f(t,\xi)=\sigma_{\mathbf{a}}(X)f(t,\xi).\]
	\begin{lemm}\label{ss}
		\begin{enumerate}
			\item   Any simple module in $\mathcal{V}_{\mathbf{a}}^{K }$ is isomorphic to $A ^{\mathbf{a}}.$
			\item The category $\mathcal{V} ^{K}_{\mathbf{a}}$ is semisimple.
		\end{enumerate}
		
	\end{lemm}
	\begin{proof}
		\begin{enumerate}
			\item We see that any submodule of $A^{\mathbf{a}}$ must contain $1$, which generates $A^{\mathbf{a}}.$ So $A^{\mathbf{a}}$ is a simple $ 
			K$-module. Let $V$ be a simple module in $\mathcal{V}_{\mathbf{a}}^{K}$. Then there exists a non-zero $v\in V$ such that $\frac{\partial}{\partial t_i}v=a_iv$ and $\frac{\partial}{\partial \xi_j}v=0$ for all $1\leq i\leq m$ and $1\leq j\leq n$. Then the map $V\rightarrow A^a,\; t^m\xi_Iv\mapsto t^m\xi_I$ defines a $K$-module isomorphism between $V$ and $A^a$. 
			\item Since $A^{\mathbf{a}}$ is equivalent to $A$ up to the superalgebra automorphism $\sigma_{\mathbf{a}}$, we  have\\ $\text{Ext}^1_{K}(A^{\mathbf{a}},A^{\mathbf{a}})\cong \text{Ext}^1_{K}(A,A)$. By the first part of this lemma, it is enough to show that $\text{Ext}^1_{K}(A,A)=0$.  
            
            Therefore, we need to show that every short exact sequence $0\rightarrow A\xrightarrow{f} M\xrightarrow{g} A\rightarrow 0$ splits. 
			Take $m\in M$ such that $g (m)=1$.\\

			\textbf{Claim:} There exists $b\in A$ such that $\partial(m+f(b))=0$ for all $\partial \in \triangle$. \\
			We know that $\partial (m)\in f(A)$ for all $\partial \in \triangle$. Now for any $\frac{\partial}{\partial \xi_j}\in \triangle_{0,n}$, if $\frac{\partial}{\partial \xi_j} m \neq 0$, then it is clear that $m\neq \xi_j\frac{\partial}{\partial \xi_j} m$, otherwise $g(m)=g(\xi_j\frac{\partial}{\partial \xi_j} m)=\xi_j\frac{\partial}{\partial \xi_j} g(m)=\xi_j\frac{\partial}{\partial \xi_j}\cdot 1=0$, a contradiction to the fact $g(m)=1$. Now it is easy to see that $\frac{\partial}{\partial \xi_j} (m-\xi_j\frac{\partial}{\partial \xi_j} m)=0$ and $\xi_j\frac{\partial}{\partial \xi_j} m\in f(A)$. Repeating the process, we can find $b_1\in A $ such that $\partial(m+f(b_1))=0$ for all $\partial\in \triangle_{0,n}$.
			Now instead of taking $m$, let us take $m_1=m+b_1$. We see that $\partial m_1\in f(A),\; \forall \;\partial \in \triangle_{m,0}$ and $\partial m_1=0,\; \forall \partial \in \triangle_{0,n}.$
			Suppose $\frac{\partial}{\partial t_1} m_1\neq 0$ and let $k\in \N$ be the minimal such that $(\frac{\partial^{k+1}}{\partial t_1^{k+1}}) m_1=0.$ Now we see that           
             \[m_1\neq (t_1\frac{\partial}{\partial t_1}m_1-\frac{1}{2!}t_1^2(\frac{\partial^2}{\partial t_1^2})m_1-\cdots + \frac{(-1)^{k-1}}{k!}t_1^k(\frac{\partial^k}{\partial t_1^k})m_1),\] otherwise $g(m_1)=0$, which is a contradiction. 
             
              Now we have the following equality in the Weyl superalgebra
            \[\frac{\partial}{\partial t_1}\cdot t_1^p(\frac{\partial^p}{\partial t_1^p})=pt_1^{p-1}(\frac{\partial^{p}}{\partial t_1^{p}})+t_1^p(\frac{\partial^{p+1}}{\partial t_1^{p+1}}),\]
             which gives us
             \begin{equation}\label{ann}
                 \frac{\partial}{\partial t_1}(m_1-t_1\frac{\partial}{\partial t_1}m_1+\frac{1}{2!}t_1^2(\frac{\partial^2}{\partial t_1^2})m_1-\cdots + \frac{(-1)^k}{k!}t_1^k(\frac{\partial^k}{\partial t_1^k})m_k)=0.
             \end{equation}
           
			Repeating the process, we have our claim and therefore $M\cong A\oplus A$. Hence $\mathcal{V}_{\mathbf{a}}^{K}$ is semisimple. 
 \end{enumerate}  
	\end{proof}
	For any $K$-module $P$ and any $T$-module $V$, the tensor product $T(P,V)=P\otimes V$ becomes an $AW$-module  under the map $\pi_1$. When $V$ is finite-dimensional we see that $T(A^{\mathbf a},V)\in \Omega_{\mathbf a}^{ AW_{m,n}}$ and the Whittaker vector space $Wh_{\mathbf a}(T(A^{\mathbf a},V))=V$.
	
     Denote by $T-$mod the category of finite-dimensional $T$-modules. 
     
	\medskip

	\begin{lemm}\label{A1}(cf. \cite{LX}).
 Given any $\mathfrak{gl}(m,n)$-module $V$, we can define an $AW$-module structure on $T(A^{\mathbf a},V)=A^{\mathbf a}\otimes V$ by the following actions: 
	\[t^\alpha \xi_I  \frac{\partial}{\partial t_i}\cdot (p\otimes v)=(t^\alpha \xi_I  \frac{\partial}{\partial t_i}\cdot p)\otimes v+\sum_{k=1}^m\alpha_k(t^{\alpha-e_k}\xi_I\cdot p)\otimes (E_{k,i}v)+(-1)^{|I|-1}\sum_{k=1}^n(\frac{\partial}{\partial \xi_k}(t^\alpha\xi_I)\cdot p)\otimes (E_{m+k,i}v),\]
	\[t^\alpha \xi_I  \frac{\partial}{\partial \xi_j}\cdot (p\otimes v)=(t^\alpha \xi_I  \frac{\partial}{\partial \xi_j}\cdot p)\otimes v+\sum_{k=1}^m\alpha_k(t^{\alpha-e_k}\xi_I\cdot p) \otimes (E_{k,m+j}v)+\]
	\[(-1)^{|I|-1}\sum_{k=1}^n(\frac{\partial}{\partial \xi_k}(t^\alpha\xi_I)\cdot p) \otimes (E_{m+k,m+j}v)\]
	\[t^\alpha \xi_I\cdot (p\otimes v)=(t^\alpha \xi_I\cdot p)\otimes v.\]
\end{lemm}

\medskip

  Let $\mathfrak{g}$ be a Lie superalgebra (or an associative superalgebra),  $M$ a $\mathfrak{g}$-module. For any finite subset $F$ of $\mathfrak{g}$ set $\text{Ann}_M(F):=\{x\in M: f.x=0,\; \forall f\in F\}$. We will simply write  $\text{Ann}(F)$   when the module is clear from the context. 

      \begin{lemm}\label{l43} 
        For any $M\in \Omega_{\mathbf a}^{ AW_{m,n}}$, the Whittaker subspace $Wh_{\mathbf a}(M)$ is a $T$-module and $$T(A^{\mathbf a},Wh(M))\cong M.$$
	\end{lemm}
    
	\begin{proof}
The $K\otimes \mathcal{U}(T)$-module structure on $M$ is given  via the isomorphism $\pi_1$. As $[K,T]=0$, the space $Wh(M)$ is a $T$-module.\\

	\textbf{Claim}: The map $\psi: A^\mathbf{a}\otimes Wh_a(M)\rightarrow M,$  $t^r\xi_I\otimes v\mapsto t^r\xi_Iv,\; r\in \Z^m_{\geq 0},\; I \subset \{1,2,\dots, n\}, $ is an isomorphism.
    
        \medskip

        \textbf{I. Surjectivity of $\psi$:}

         \textbf{Subclaim 1:} We have the following: 
        \begin{equation}\label{e}
            M=\text{Ann}(\triangle_{0,n})+\psi(\Lambda(n)\otimes \text{Ann}(\triangle_{0,n})).
        \end{equation}
        
     It is easy to see,  $\text{Ann}(\triangle_{0,n})+\psi(\Lambda(n)\otimes \text{Ann}(\triangle_{0,n}))\subseteq M.$ 
        \medskip
        
                First, we will prove that
        \begin{equation}\label{e1}
            M=\text{Ann}( \frac{\partial }{\partial \xi_1})+\psi( \xi_1\otimes \text{Ann}( \frac{\partial }{\partial \xi_1}))
        \end{equation}
        
         It is easy to see $\text{Ann}( \frac{\partial }{\partial \xi_1})+\psi( \xi_1\otimes \text{Ann}( \frac{\partial }{\partial \xi_1}))\subseteq M.$
        \
        Let $x\in M$ such that $\frac{\partial }{\partial \xi_1}x= 0$, then it is trivially true. Assume $\frac{\partial }{\partial \xi_1}x\neq 0$, then it is easy to see that $\frac{\partial }{\partial \xi_1} (x-\xi_1\frac{\partial }{\partial \xi_1}x)=0$. As $\frac{\partial }{\partial \xi_1}\frac{\partial }{\partial \xi_1}x=0$, we have $M=\text{Ann}( \frac{\partial }{\partial \xi_1})+\psi( \xi_1\otimes \text{Ann}( \frac{\partial }{\partial \xi_1}))$. In similar fashion, we get $ \text{Ann}( \frac{\partial }{\partial \xi_1})\subseteq \text{Ann}(\frac{\partial }{\partial \xi_1}, \frac{\partial }{\partial \xi_2})+\psi( \xi_2\otimes \text{Ann}( \frac{\partial }{\partial \xi_1},\frac{\partial }{\partial \xi_2} ))$, putting this on equation \ref{e1} we get,

        \begin{equation}
\begin{split}
M & = \text{Ann}(\frac{\partial }{\partial \xi_1}, \frac{\partial }{\partial \xi_2})+\psi( \xi_1\otimes(  \text{Ann}( \frac{\partial }{\partial \xi_1},\frac{\partial }{\partial \xi_2} )+\psi(\xi_2\otimes \text{Ann}(\frac{\partial }{\partial \xi_1}, \frac{\partial }{\partial \xi_2})))) \\
 & = \text{Ann}(\frac{\partial }{\partial \xi_1}, \frac{\partial }{\partial \xi_2})+\psi( \xi_1\otimes  \text{Ann}( \frac{\partial }{\partial \xi_1},\frac{\partial }{\partial \xi_2} ))+\psi( \xi_2\otimes  \text{Ann}( \frac{\partial }{\partial \xi_1},\frac{\partial }{\partial \xi_2} ))+\psi( \xi_1\xi_2\otimes  \text{Ann}( \frac{\partial }{\partial \xi_1},\frac{\partial }{\partial \xi_2} )). 
\end{split}
\end{equation}

Now continuing this process we will have equation \ref{e}. Now we will show the following

\begin{equation}\label{hta1}
\text{Ann}(\triangle_{0,n})\subseteq \text{Ann}(\triangle_{0,n}, \frac{\partial}{\partial t_1}-a_1)+\psi(\C[t_1] \otimes \text{Ann}(\triangle_{0,n}, \frac{\partial}{\partial t_1}-a_1))
\end{equation}

Given any $x\in M$, we define $Ht_{a_1}(x):=\text{min}\{k|(\frac{\partial}{\partial t_1}-a_1)^{k+1}.x=0\}$. We will prove Equation \ref{hta1} by induction on height of elements of $\text{Ann}(\triangle_{(0,n)})$. 

\medskip

Equation \ref{hta1} is trivially true for any height zero elements. Let $x\in \text{Ann}(\triangle_{(0,n)})$ be an element of height $1$, then we see that $x-t(\frac{\partial}{\partial t_1}-a_1)x\in \text{Ann}(\triangle_{0,n}, \frac{\partial}{\partial t_1}-a_1)$ and hence Equation \ref{hta1} is true for any elements upto height $1$. Now we assume that equation is true for all elements of height $k-1$ and $x\in \text{Ann}(\triangle_{0,n})$ is an element of height $k$. Then (similarly to  Equation \ref{ann}) it is easy to see that 

\begin{equation*}
   x-t_1(\frac{\partial}{\partial t_1}-a_1)x+\frac{1}{2!}t_2^2(\frac{\partial}{\partial t_1}-a_1)^2x-\frac{1}{3!}t_1^3(\frac{\partial}{\partial t_1}-a_1)^3x+\cdots +\frac{(-1)^k}{k!}t_1^k(\frac{\partial}{\partial t_1}-a_1)^kx\in  \text{Ann}(\triangle_{0,n}, \frac{\partial}{\partial t_1}-a_1).
\end{equation*}

Now by induction assumption we have     
 \begin{equation*}
     \begin{split}
         t_1(\frac{\partial}{\partial t_1}-a_1)x-\frac{1}{2!}t_2^2(\frac{\partial}{\partial t_1}-a_1)^2x+\frac{1}{3!}t_1^3(\frac{\partial}{\partial t_1}-a_1)^3x+\cdots +\frac{(-1)^{k-1}}{k!}t_1^k(\frac{\partial}{\partial t_1}-a_1)^kx
         \\ \in  \text{Ann}(\triangle_{0,n}, \frac{\partial}{\partial t_1}-a_1)+\psi(\C[t_1] \otimes \text{Ann}(\triangle_{0,n}, \frac{\partial}{\partial t_1}-a_1)).
     \end{split}
 \end{equation*}
 
Therefore Equation \ref{hta1} is true for all elements of height $k$ and hence by induction it is true for all elements of $\text{Ann}(\triangle_{0,n}).$
\medskip

Now in a similar fashion we can prove the following
\begin{equation}\label{hta2}
\text{Ann}(\triangle_{0,n}, \frac{\partial}{\partial t_1}-a_1)\subseteq \text{Ann}(\triangle_{0,n}, \frac{\partial}{\partial t_1}-a_1,\frac{\partial}{\partial t_1}-a_2)+\psi(\C[t_2]\otimes  \text{Ann}(\triangle_{0,n}, \frac{\partial}{\partial t_1}-a_1,\frac{\partial}{\partial t_1}-a_2)) 
\end{equation}
Substituting \ref{hta2} in \ref{hta1} we get 

\begin{equation}
\text{Ann}(\triangle_{0,n})\subseteq \text{Ann}(\triangle_{0,n}, \frac{\partial}{\partial t_1}-a_1,\frac{\partial}{\partial t_2}-a_2)+\psi(\C[t_1,t_2] \otimes \text{Ann}(\triangle_{0,n}, \frac{\partial}{\partial t_1}-a_1,\frac{\partial}{\partial t_2}-a_2))
\end{equation}

Continuing this process we get

\begin{equation}\label{htam}
   \text{Ann}(\triangle_{0,n})\subseteq \text{Wh}_a(M)+\psi(\C[t_1,\cdots,t_m]\otimes \text{Wh}_a(M)) 
\end{equation}

Putting Equation \ref{htam} in  \ref{e}, we get 

\begin{equation}
    M=\text{Wh}_a(M)+\psi(A\otimes M),
\end{equation}

and hence we have surjectivity of $\psi$.

         {\textbf{II. Injectivity of $\psi$:}}
         
         We know by  Lemma \ref{ss}(1), $A^a$ is a simple $K$-module, and hence it is strictly simple. We assume 
         \[w=\sum_{(r,I)\in F_w}t^r\xi_I\otimes v_{(r,I)}\in \text{Ker}(\psi),\]
         where $F_w$ is a finite subset of $\Z_{\geq 0}^m\times \mathcal{P}(n)$ with $\mathcal{P}(n)$ as set of all subsets of $\{1,\cdots,n\}$. Now by Density theorem we know that there exist $X\in K$ such that $X.t^r\xi_I=1$ and $X.t^s\xi_J=0$ for a fixed $(r,I)\in F_w$ and all $(s,J)\in F_w\setminus\{(r,I)\}.$ Then we have $X.\psi(w)=\psi(X.w)=v_{(r,I)}=0$ and hence we have $\psi$ is injective.
        
	\end{proof}
	
	\begin{theo}\label{equiv}
		The functor 
		\[\mathcal{F}: T-mod\rightarrow \Omega_{\mathbf a}^{ AW_{m,n}},\]
		\[V \mapsto T(A^{\mathbf a},V),\]
		is an equivalence of categories. 
	\end{theo}
	\begin{proof}
		By the Schur's Lemma, $End_K(A^{\mathbf a})=\C$. By Lemma \ref{A1} and Lemma \ref{l43}, for any $M\in \Omega_{\mathbf a}^{AW_{m,n}}$ there exists a $T$-module $Wh_{\mathbf a}(M)$ such that $\mathcal{F}(Wh_{\mathbf a}(M))\cong M$. We have the following isomorphism $$Hom_{K\otimes \mathcal{U}(T)}(A^{\mathbf a}\otimes V, A^{\mathbf a}\otimes W)\cong Hom_{\mathcal{U}(T)}(V,W),$$ which implies that the homomorphism
		$\mathcal{F}_{V,W}:Hom_T(V,W) \rightarrow Hom_{\widetilde{W}}(\mathcal{F}(V),\mathcal{F}(W))$ is an isomorphism, for all $V,W\in T-mod$. Therefore $\mathcal{F}$ is an equivalence of categories. 
	\end{proof}
	
	Applying Lemma \ref{ss}, Lemma \ref{A1} and Theorem \ref{equiv} we obtain  the following theorem.
    
 \begin{theo}\label{EXW}
  Suppose $M$ is a simple Whittaker $AW$-module of type ${\mathbf a}$ with  finite-dimensional $Wh_{\mathbf a}(M)$. Then $M\cong T(A^{\mathbf a},V)$ for some finite-dimensional simple $\mathfrak{gl}(m,n)$-module $V$. 
 \end{theo}

	\section{Whittaker modules in $\Omega_{\mathbf{a}}^{ W}$ }\label{W}

	In this Section we will prove our main theorem. First, we make the following  observation.
    \begin{remark}
     For any $M\in \Omega_{\mathbf a}^{W}$  generated by $Wh_{\mathbf a}(M)$, we have dim$(\widetilde{Wh_{\mathbf a}}(M))< \infty$. 
    \end{remark}
    \begin{proof} 
        
        Assume that the dimension of a fixed generalized Whittaker space is infinite. 
         Since $(\frac{\partial}{\partial {\xi_j}})^2.v=0$ for all $v\in M$, then it contains  an infinite direct sum 
         of indecomposable modules over $\mathbb C[\frac{\partial}{\partial \xi_i}]$ (corresponding to Jordan blocks with eigenvalue zero). Hence, the subspace of eigenvectors of 
         $\frac{\partial}{\partial \xi_i}$ with eigenvalue zero in $\widetilde{Wh_{\mathbf a}}(M)$   is infinite dimensional. Denote this subspace by $S_{\mathbf a}^i$. Consider now the subspace $ \frac{\partial}{\partial \xi_j}(S_{\mathbf a}^i)\subset \widetilde{Wh_{\mathbf a}}(M)$ for $j\neq i$. We see that either $S_{\mathbf a}^i$ or $ \frac{\partial}{\partial \xi_j}(S_{\mathbf a}^i)$ is annihilated by both $\frac{\partial}{\partial \xi_i}$ and $\frac{\partial}{\partial \xi_j}$, as these operators anticommute. Moreover, 
         if $ \frac{\partial}{\partial \xi_j}(S_{\mathbf a}^i)$ is finite dimensional then $S_{\mathbf a}^i\cap Ker(\frac{\partial}{\partial \xi_j})$ is infinite dimensional. 
         We see by induction that $\widetilde{Wh_{\mathbf a}}(M)$ contains an infinite dimensional subspace annihilated by all $\frac{\partial}{\partial {\xi_j}}$. We conclude that dim$(Wh_{\mathbf a}(M))= \infty$, which is a contradiction.

    \end{proof}
    
    From now on we will assume that ${\mathbf a}\in \C^m$ is a non-singular vector. 
    
    Any Whittaker module over $W_{m,n}$ can be considered as Whittaker module of  $W_{m,0}$ by restriction.  Following  \cite[Lemma 4.1]{ZL}, it is easy to see that any Whittaker module over the Witt superalgebra $W$ is a free $\mathcal{U}(\mathfrak{h}_{m,0})$-module. Moreover, it is generated by generalized Whittaker vectors  as a $W_{m,0}$-module. In particular we have:
	\begin{lemm}\label{CartanF}
		Suppose $M\in \Omega_{\mathbf a}^W$  is generated by $Wh_{\mathbf a}(M)$ as a $W$-module, then\\ $M=\mathcal{U}(\mathfrak{h}_{m,0})\widetilde{Wh_{\mathbf a}}(M)$.
		Moreover, if $\{v_i:i \in \Lambda\}$ is a basis of $\widetilde{Wh_{\mathbf a}}(M)$, then $\{h^rv_i:r \in \Z^m_{\geq 0}, i \in \Lambda\}$ is a basis of $M$, where $h^r=(t_1\frac{\partial}{\partial t_1})^{r_1}(t_2\frac{\partial}{\partial t_2})^{r_2}\cdots (t_m\frac{\partial}{\partial t_m})^{r_m}$. 
	\end{lemm}
    
	\begin{coro}\label{Cartan free}
		Any non-trivial $M\in \Omega_{\mathbf a}^{W}$, which is generated by $Wh_{\mathbf a}(M)$ as a $W$-module, is a $\mathcal{U}(\mathfrak{h}_{m,0})$-free module of finite rank.   
	\end{coro}

	Now we will define the cover of the Whittaker $W$-module $M$. Note that $W\otimes M$ is a $\widetilde{W}$-module with the action given by
	$$x\cdot (a\otimes b)=[x,a]\otimes b+ a\otimes x \cdot b,\; y\cdot a\otimes b=ya\otimes b,$$
	where $x,a\in W,\; y\in A,\; b\in M$.
    
	Let $\theta:W\otimes M\rightarrow M$ be the linear map defined by $\theta(a\otimes b )=ab$ for $a\in W,\; b\in M$. We see that $\theta$ is a $W$-module homomorphism.
	Denote $K(M)=\{v\in \text{Ker}(\theta)/ Av\subseteq \text{Ker}(\theta)\}$. It is easy to see that $K(M)$ is a $AW$-submodule of $W\otimes M$. Set $\widehat{M}=(W\otimes M)/K(M)$. Note that $\theta$  induces a $W$-module homomorphism $\hat{\theta}:\widehat{M}\rightarrow M$. The module $\widehat{M}$ is called the cover of $M$ if $WM=M$. For any $X\in W$, $m\in M$, we denote the element $\hat{\theta}(X\otimes m)$ by $X\boxtimes m$.
	\begin{remark}
		If $M\in \Omega_{\mathbf a}^W$, then $\triangle$ acts locally finitely  on $W\otimes M$ and on $\widehat{M}$.
	\end{remark}
 
	For any $\alpha, \beta \in \Z^m_+,\; I,J \subseteq \{1,2,\dots, n\},\; r\in \N,\; 1\leq j\leq m,\; \partial, \partial^\prime \in \triangle$, we define
	\begin{equation}
	w_{\alpha,\beta, I, J}^{r,j,\partial,\partial ^{\prime}}=\sum_{i=0}^r(-1)^i {r \choose i} t^{\alpha+ (r-i) e_j}\xi_I \partial \cdot  t^{\beta +ie_j} \xi_J\partial^{\prime} \in \mathcal{U}(W).
	\end{equation}
	Then we have
	\begin{equation}
		w_{\alpha+e_j,\beta, I, J}^{r,j,\partial,\partial^\prime}-w_{\alpha,\beta+e_j, I, J}^{r,j,\partial,\partial^\prime}=w_{\alpha,\beta, I, J}^{r+1,j,\partial,\partial^\prime}.
	\end{equation}
	\begin{theo}\label{UBM}
		Suppose that $M$ is a uniformly bounded $W$-module. Then there exists $r\in \N$ such that $w_{\alpha,\beta, I, J}^{r,j,\partial,\partial^\prime}M=0$ for all $\alpha, \beta \in \Z^m_+,\; I,J \subseteq \{1,2,\dots, n\},\;  1\leq j\leq m,\; \partial, \partial^\prime \in \triangle$.
	\end{theo}
	\begin{proof}
		The proof is analogous to the proof  of \cite[Lemma 4.2]{LX}. Note that the definition of a weight module in  \cite{LX} is different from our definition of  weight module, as we only require a diagonalizable action of the Cartan subalgebra of $W_{m,0}$. Nevertheless, the argument remains valid.  Let $M$ be a $W_{m,0}$-module. By  \cite[Lemma 4.5]{XL1},  there exists $r_1\in \N$  such that $w_{\alpha,\beta,  \emptyset, \emptyset}^{r,j,\partial ,\partial^\prime}.M=0$, for $\partial ,\partial^\prime\in \triangle_{m,0}$, $j\in\{1,\cdots m\}$. Computing the Lie brackets  $[w_{\alpha,\beta,  \emptyset, \emptyset}^{r,j,\partial ,\partial^\prime}, t^s\xi_I\partial^\prime]$ in $\mathcal{U}(W)$, we see that $w_{\alpha,\beta,  \emptyset, I}^{r_1+2,j,\partial ,\partial }.M=0$ for all $\alpha, \beta\in \Z^m_{+}$ and $I\subseteq \{1,2,\dots n\}$ and $j\in \{1\cdots ,m\}$. Let $J\subseteq \{1,2,\dots n\}$. Computing the Lie bracket $[w_{\alpha,\beta,  \emptyset, I}^{r_1+2,j,\partial ,\partial }, t^\gamma \partial^{\prime}]$  in $\mathcal{U}(W)$, with $\gamma\in \Z^m_+$ and $\partial^\prime \in \triangle_{m,n}\setminus\{\partial\}$, we see that $w_{\alpha,\beta,  J, I}^{r_1+4,j,\partial ,\partial^\prime}.M=0$. Again choosing $\partial^{\prime\prime}\in \triangle_{m,n}\setminus \{\partial\}$ and computing the appropriate Lie brackets we see that $w_{\alpha,\beta,  J, I}^{r_1+4,j,\partial^{\prime \prime},\partial^\prime}.M=0$, and hence we have our result.        
     \end{proof}
    
	\begin{theo}\label{hat{M}}
		If $M\in \Omega_{\mathbf a}^W$ is generated by $Wh_{\mathbf a}(M)$ as a $W$-module, then $\widehat{M} \in \Omega_{\mathbf a}^{ AW_{m,n}}$.    
	\end{theo}
    
	\begin{proof}
		We need to prove that $\widetilde{Wh_{\mathbf a}}(\widehat{M})$ is finite dimensional.
		By Lemma \ref{CartanF} and Corollary \ref{Cartan free}, $M$ is a free $\mathcal{U}(\mathfrak{h}_{m,0})$-module of finite rank. So, $\mathfrak{W}(M)$ is a uniformly bounded weight $W$-module. The dimension of each weight space of $\mathfrak{W}(M)$ does not exceed the rank of $M$ as a free $\mathcal{U}(\mathfrak{h}_{m,0})$-module. By Theorem \ref{UBM}, there exists $r\in \N$ such that 
		$$w_{\alpha,\beta, I, J}^{r,j,\partial,\partial^\prime}\cdot \mathfrak{W}M=0,$$
		for all $\alpha, \beta \in \Z^m_+,\; I,J \subseteq \{1,2,\dots, n\},\; 1\leq j\leq m,\; \partial, \partial^\prime \in \triangle$.\\
		Now we deduce that 
		$$w_{\alpha,\beta, I, J}^{r,j,\partial,\partial^\prime}\cdot M\subseteq \bigcap_{\gamma \in \Z^m}I_{\gamma}M=0,$$
		for all $\alpha, \beta \in \Z^m_+,\; I,J \subseteq \{1,2,\dots, n\},\; 1\leq j\leq m,\; \partial, \partial^\prime \in \triangle,$
		where $I_\gamma$ is the maximal ideal of $\mathcal{U}(\mathfrak{h}_{m,0})$ generated by $h_1-\gamma_1,\dots, h_n-\gamma_n$. Here we have $\bigcap_{\gamma \in \Z^m}I_{\gamma}M=0,$ since $M$ is a finitely generated free $\mathcal{U}(\mathfrak{h}_{m,0})$-module. \\
		Now $M$ being simple, we  have $M=WM$.  We will prove by induction on $|p|=p_1+\dots+p_m$ that 
		$$(t^p\xi_I\partial \boxtimes t^s\xi_J\partial^\prime)v\in  \sum_{\substack{|q|\leq rm,\\ I_1\subseteq \{1\cdots n\},\\ \partial \in \triangle}}t^q\xi_{I_1}\partial \boxtimes M,$$
		for all $v\in M$, $p,s\in \Z_+^m,\; I,J\subseteq \{1,2,\dots,n\},\; \partial,\partial^\prime \in \triangle$.
		The result is clear for $p\in \Z^m_+$ with $|p|\leq rm$. So we assume that $|p|>rm$. Without loss of generality, we may assume that $p_1>r$. \\
        
		For any $p\in \Z_+^m$, we have 
	 
		$$\theta(\sum_{i=0}^r(-1)^i{r\choose i}t^{p-ie_1}\xi_I\partial\otimes (t^{s+ie_1}\xi_J\partial ^\prime \xi_Jv))$$	 
		$$=\sum_{i=0}^r(-1)^i{r\choose i}t^{p-ie_1}\xi_I\partial  t^{s+ie_1}\xi_J\partial ^\prime \xi_Jv$$
		$$=w_{p,s,I,J}^{r,1,\partial, \partial^\prime}=0.$$
		By the construction of the cover $\widehat{M}$, this means that 
		
		$$\sum_{i=0}^r(-1)^i{r\choose i}(t^{p-ie_1}\xi_I\partial\otimes t^{s+ie_1}\xi_J\partial ^\prime \xi_J)v \in K(M).$$
		Then we get 
		\[(t^p\xi_I\partial  \boxtimes t^s \xi_J \partial^\prime) v=\sum_{i=1}^r(-1)^{i-1}{r\choose i}(t^{p-ie_1}\xi_I\partial\boxtimes t^{s+ie_1}\xi_J\partial ^\prime \xi_J v)\in \widehat{M},\] which belongs to 
		\[ \sum_ {\substack{|q|\leq rm,\\ I_1\subseteq \{1\cdots n\},\\ \partial \in \triangle}}t^q\xi_{I_1}\partial\boxtimes M\]
by the induction hypothesis. 
\medskip

Therefore $\widehat{M}$ is a finitely generated $\mathcal{U}(\mathfrak{h}_{m,0})$-module. By Lemma \ref{CartanF} the space $\widetilde{Wh_{\mathbf a}}(\widehat{M})$ is finite dimensional, since otherwise $\widehat{M}$ is a free $\mathcal{U}(\mathfrak{h}_{m,0})$- module of infinite rank, which is a contradiction.	
	\end{proof}

    Now we are ready to prove our main result.
    
    \begin{theo}\label{MT}
		
		If $M\in \Omega_{\mathbf a}^W$ is a simple module, then $M$ is isomorphic to a simple quotient of $T(A^{\mathbf a}, V)$ for some finite-dimensional simple $\mathfrak{gl}(m,n)$-module $V$.		
	\end{theo}
    
	\begin{proof}
		By Theorem \ref{hat{M}}, we have a $W$-module epimorphism $\phi:M_1\rightarrow M$, where $M_1$ is a  Whittaker $\widetilde{W}$-module of type $\mathbf{a}$ with finite-dimensional $\text{Wh}_{\mathbf a}(M_1)$. Without loss of generality, we assume that dimension of  $Wh_a(M_1)$ is minimal. This will force $M_1$ to be irreducible, otherwise we can find one maximal submodule of $M_1$, say $M_2$ over $\widetilde{W}$. Then by Corollary \ref{Cartan free}, we see that both $M_2$ and $M_1/M_2$ are $\mathcal{U}(\mathfrak{h})$-free module of rank less than $\text{dim Wh}_{\mathbf a}(M_1)$. Now $M$ being simple $W$-module, either $\phi(M_2)=0$ or $\phi(M_2)=M$. So either $M_1/M_2$ or $M_2$ has a simple $W$-quotient module isomorphic to $M,$ which contradicts the minimality of $M_1$. Now by Theorem \ref{EXW}, we have $M_1\cong T(A^{\mathbf a},V)$ for a finite-dimensional $\mathfrak{gl}(m,n)$ module $V$ and hence $M$ is isomorphic to a simple quotient of $T(A^{\mathbf a},V)$.
	\end{proof}
    
The simplicity of these $W$-modules and their irreducible quotients were described in  \cite{XW} (see \cite{LLZ} for Lie algebra case). Together with these results, Theorem \ref{MT} provides a classification of all simple non-singular Whittaker modules for $W_{m,n}$ with finite-dimensional Whittaker subspaces.

	\section{Acknowledgments}
\noindent V.Futorny is supported by the NSF of China (12350710178 and 12350710787).

\end{document}